\documentclass{amsart}

\usepackage[utf8]{inputenc}
\usepackage{amscd}
\usepackage{amsmath}
\usepackage{amssymb}
\usepackage{amsthm}
\usepackage{epsf}
\usepackage{verbatim}
\usepackage{graphicx}
\usepackage{placeins}
\usepackage{enumitem}
\input epsf.tex

\usepackage[letterpaper,top=3cm,bottom=2cm,left=3cm,right=3cm,marginparwidth=1.75cm]{geometry}

\newtheorem{theorem}{Theorem}[section]
\newtheorem{lemma}[theorem]{Lemma}
\newtheorem{corollary}[theorem]{Corollary}

\newtheorem{exercise}[theorem]{Exercise}
\newtheorem{example}{Example}[section]

\let\oldexercise\exercise
\renewcommand{\exercise}{\oldexercise\normalfont}

\let\oldexample\example
\renewcommand{\example}{\oldexample\normalfont}

\newcommand{\updownle}{\mathbin{\rotatebox[origin=c]{270}{$\le$}}}

\title{Multi-crossing Braids}
\date{May 8, 2018}
\author{Daishiro Nishida}

\begin{document}

\begin{abstract}
Traditionally, knot theorists have considered projections of knots where there are two strands meeting at every crossing. A multi-crossing is a crossing where more than two strands meet at a single point, such that each strand bisects the crossing. In this paper we generalize ideas in traditional braid theory to multi-crossing braids. Our main result is an extension of Alexander's Theorem. We prove that every link can be put into an $n$-crossing braid form for any even $n$, and that every link with two or more components can be put into an $n$-crossing braid form for any $n$. We find relationships between the $n$-crossing braid indices, or the number of strings necessary to represent a link in an $n$-crossing braid.
\end{abstract}

\maketitle


\section{Introduction}
\label{sec:intro}

In traditional knot theory, knots are drawn in a projection where there are two strands passing over each other at every crossing. An \emph{$n$-crossing} is a crossing where there are $n$ strands meeting at one point, with each strand bisecting the crossing. We call this crossing a \emph{multi-crossing} if $n>2$, and we call the traditional type ($n=2$) a \emph{double crossing}. The strands are labeled with the levels $1, \dots n$ from the top.

In~\cite{triple crossing}, Adams proved that every link has an $n$-crossing projection for all $n \ge 3$. This fact allows us to generalize notions in traditional knot theory to their multi-crossing versions. For example, the \emph{crossing number} $c(L)$, the minimum number of crossings in any double crossing projection of the link $L$, generalizes to the \emph{multi-crossing number} $c_n(L)$, which is the minimum number of crossings in any $n$-crossing projection of the link $L$. This gives us an infinite spectrum of crossing numbers that can be explored.

In this paper we similarly generalize ideas from braid theory to multi-crossing braids. Alexander's Theorem states that every link can be put into a braid form \cite{alexander}. This Theorem has since been proved in several ways \cite{yamada} \cite{braid index bound}. Is it true that every link can be put into an $n$-crossing braid, or a braid where every crossing is an $n$-crossing? If true, we can generalize the notion of \emph{braid index} $\beta(L)$, which is the minimum number of strings needed to represent the link $L$ as the closure of a double crossing braid. Can we define the $n$-crossing braid indices $\beta_n(L)$, and what are their properties?

In Section~\ref{sec:evenAlexander} we prove a version of Alexander's Theorem~\cite{alexander} for even multi-crossing braids. Specifically, we prove that every link can be represented as a closed $n$-crossing braid, for all even $n$. In Section~\ref{sec:tripleAlexander} we consider an equivalent of Alexander's Theorem for triple crossing braids. We prove that every link with at least two components can be represented as a closed triple-crossing braid. In Section~\ref{sec:oddAlexander} we extend this result to all odd $n$. Finally, in Section~\ref{sec:braidIndices} we find relationships between the $n$-crossing braid indices.

This paper is a part of a senior thesis completed at Williams College. I would like to thank my advisor Professor Colin Adams for his guidance throughout this year. Thanks to Daniel Vitek for suggesting this problem. He had proved Theorem~\ref{thm:evenAlexander} for $n=6,10,14,18,22,26$ using Lemma~\ref{lem:levelPosition} and computation by Mathematica. The idea to convert the problem into looking at the corresponding permutations is due to him, as is the proof of Lemma~\ref{lem:levelPosition}.

\section{Even Multi-crossing Braids}
\label{sec:evenAlexander}

In this section, we prove a result similar to Alexander's Theorem for $n$-crossing braids, where $n$ is even. Specifically, we prove the following.

\begin{theorem}
\label{thm:evenAlexander}
Every link can be represented as a closed $n$-crossing braid, for all even $n$.
\end{theorem}

We prove this theorem by starting with a link in double crossing braid form, and finding an isotopy to make it a sequence of $n$-crossings.

\subsection{Level position}

Given a collection of $m$ strings, label them $1, \dots, m$ from the left. We can then think of each crossing that occurs in these $m$ strings as a permutation of the strings. Specifically, we can define a homomorphism $\phi:B_n \rightarrow S_n$ by $\phi(\sigma_i) = \phi(\sigma_i^{-1}) = (i,i+1)$, and extend by linearity. A double crossing corresponds to a transposition $(i, i+1)$. An $n$-crossing corresponds to a permutation of the form $\pi_j = (j,j+n-1)(j+1,j+n-2)\cdots$; we call any permutation of this form a \emph{crossing permutation}. Note that there are $m-n+1$ possible $n$-crossing permutations, corresponding to each $j$ where $1 \le j \le m-n+1$.

First we show a lemma that allows us to ignore the levels of each crossing and focus on the images under $\phi$. We say a sequence of crossings is \emph{disjoint} if each string is switched by at most one crossing in the sequence. In this section we only need this lemma for $s=2$, but we state a general version which can be used in later sections for other values of $s$.

\begin{lemma}
\label{lem:levelPosition}
Let $\alpha$ be a sequence of disjoint $s$-crossings. Suppose that a product of $n$-crossing permutations in $S_m$ produces $\phi(\alpha)$. Then there exists a sequence of $n$-crossings over a $m$-string braid which produces $\alpha$.
\end{lemma}

\begin{proof}
Consider the sequence of $n$-crossing permutations that produces $\phi(\alpha)$. We want to show that we can choose the levels of the corresponding $n$-crossings appropriately so that the result is equivalent to $\alpha$.

We do this by placing the $m$ strands on different heights. Choose an $s$-crossing in $\alpha$. We place the $s$ strands of this crossing in the $s$ highest levels, according to their levels in the $s$-crossing. We continue this process by choosing a new $s$-crossing in $\alpha$, and placing the $s$ strands of this crossing in the next $s$ highest levels. Once we have exhausted $s$-crossings in $\alpha$, the remaining strands can be placed in any order. The heights are well defined since the crossings are disjoint.

For each permutation corresponding to an $n$-crossing, we want to assign levels to the strands to make it into an $n$-crossing. We can simply choose the levels of the strands in the order of the heights assigned above. This will mean that each strand always stays on the same level, and that each $n$-crossing can be untangled easily (Fig.~\ref{fig:diagonalBox}). We call this a \emph{level position} of the braid.

\begin{figure}[ht]
	\centering
	\includegraphics[scale=0.2]{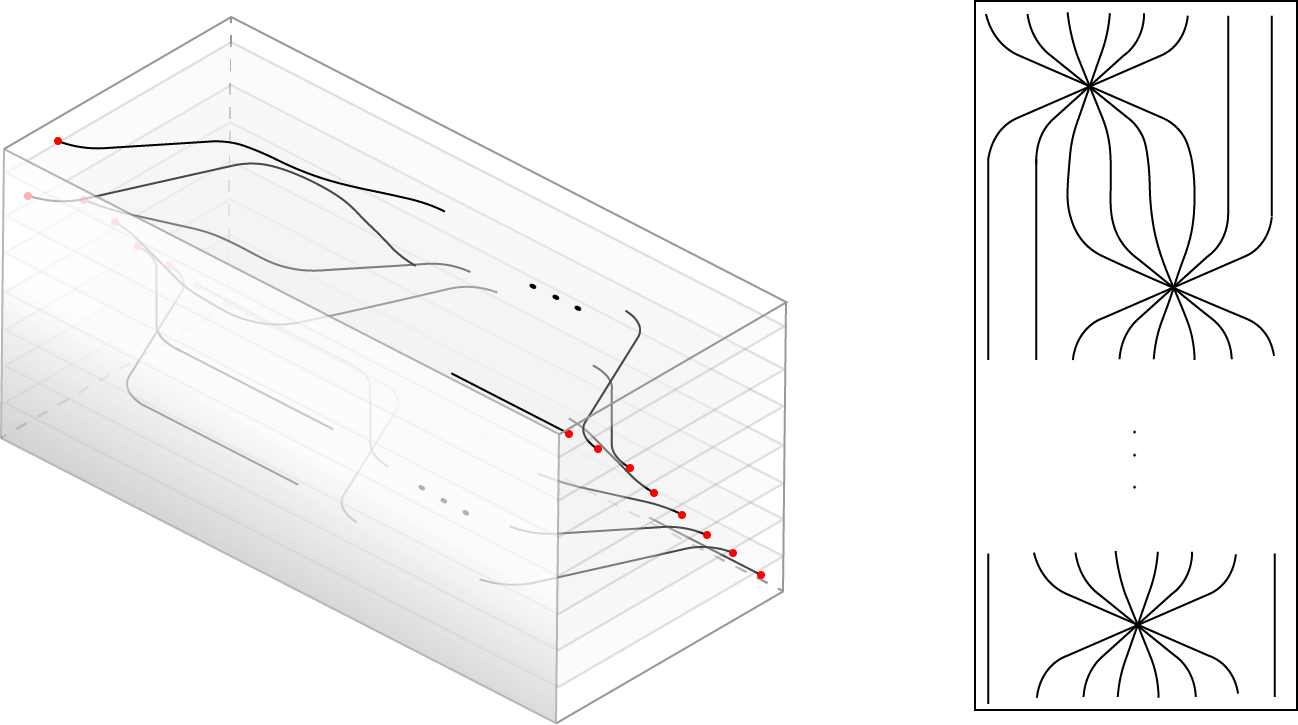}
	\caption{The left picture shows how we might use 6-crossings to obtain a double crossing which switches the first two strings, when the second string is the overstrand. The remaining strands have been ordered from left to right. The right picture is a view from above.}
	\label{fig:diagonalBox}
\end{figure}

Then, once we have achieved $\phi(\alpha)$ we can pull the strings taut, so that the only crossings that are left are the $s$-crossing in $\alpha$ that we were looking for (Fig.~\ref{fig:diagonalBoxAfterPull}).
\begin{figure}[ht]
	\centering
	\includegraphics[scale=0.2]{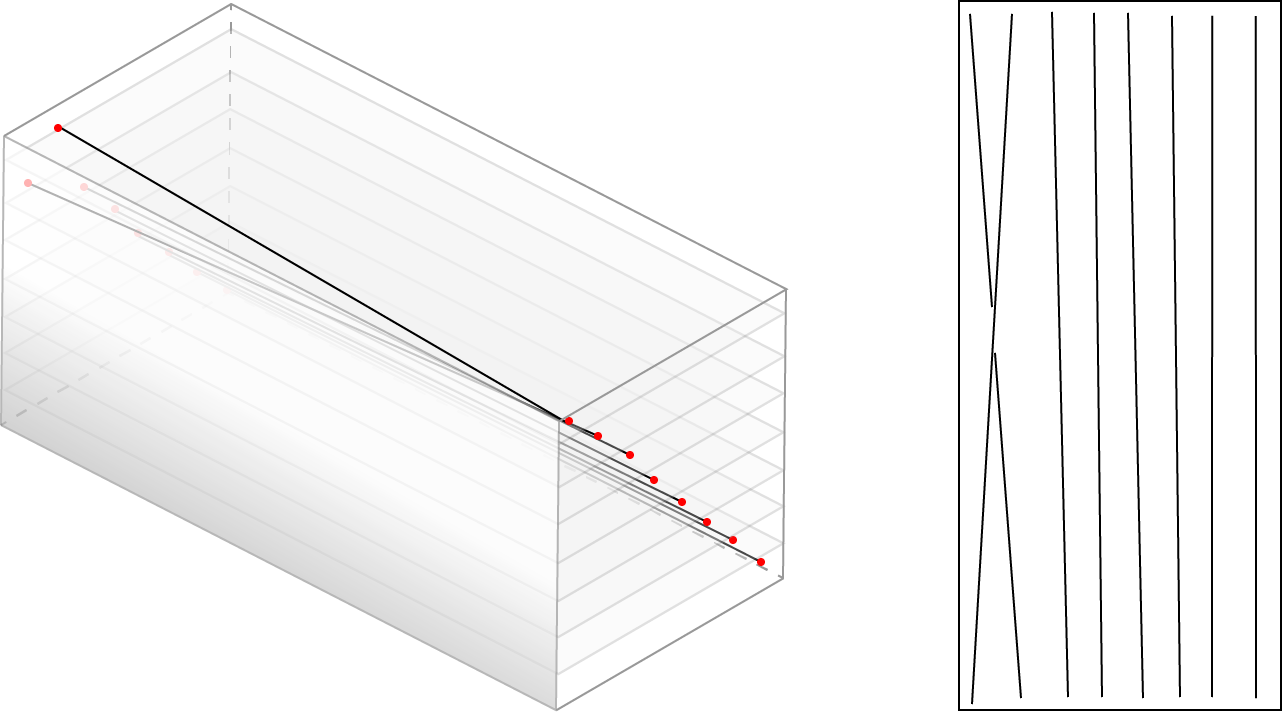}
	\caption{What happens after the strings are pulled taut. All the 6-crossings disappear and one double crossing remains.}
	\label{fig:diagonalBoxAfterPull}
\end{figure}
\end{proof}

By this lemma, it suffices to show that we can use $n$-crossing permutations to obtain double crossing permutations.

\subsection{Conjugation of permutations}

Observe that for permutations $\pi, \sigma \in S_m$, where
\[
\pi = (i_1, i_2, \dots, i_r) \cdots (i_s, i_{s+1}, \dots, i_t)
\]
in cycle notation, the conjugate of $\pi$ by $\sigma$ has a nice form:
\[
\sigma \pi \sigma^{-1} = (\sigma(i_1), \sigma(i_2), \dots, \sigma(i_r))
\cdots (\sigma(i_s), \sigma(i_{s+1}), \dots, \sigma(i_t)).
\]
We consider the case when $\sigma$ is a crossing permutation. Then it is its own inverse, so $\sigma \pi \sigma^{-1} = \sigma \pi \sigma$. Thus, if we multiply both sides of $\pi$ by $\sigma$, we are essentially switching elements that appear in $\pi$ according to $\sigma$. Note that taking the conjugate does not change the cycle type, which is to say it keeps the number of cycles and the length of each cycle constant. Also note that taking conjugates by the crossing permutation $\pi_i$ can only affect elements between $i$ and $i+n-1$, and must leave all other elements fixed. Finally, observe that we can reverse this process since $\sigma$ is its own inverse; if we can obtain $\pi$ by taking conjugates of $\pi'$, then we can obtain $\pi'$ by taking conjugates of $\pi$ in reverse order.

\subsection{Obtaining permutations of the same cycle type}

Using this idea, we can start with $\pi_1$, and take conjugates by some $\pi_j$ to obtain different permutations with the same cycle type. In fact, we can show that for sufficiently large $m$, we can obtain any permutation with the same cycle type. Note that we can always make sure that we have enough strands (i.e. that $m$ is large enough) by taking stabilizations.

First we present several lemmas which will be helpful in proving this result. In the proofs of these lemmas we repeatedly take conjugates by crossing permutations. Recall that taking conjugates by the crossing permutation $\pi_i$ can only affect elements between $i$ and $i+n-1$, and must leave all other elements fixed. We can keep track of which entries of a permutation are affected by conjugation, by checking which entries lie within or outside the range from $i$ to $i+n-1$.

The first lemma shows how to obtain a permutation which sends 1 to some $N$.

\begin{lemma}
\label{lem:moveMaxPerm}
Let $n$ be even. Then, for any $N$ with $1 \le N \le \frac{3n}{2}$, there exists a sequence of $n$-crossing permutations over $S_{3n/2}$ whose product sends 1 to $N$.
\end{lemma}

\begin{proof}
First observe that over $S_{3n/2}$, we have the $n$-crossing permutations $\pi_1, \dots, \pi_{n/2+1}$.

We consider permutations which send $r$ to some $r+j$. We call this incrementing $r$ by $j$.

For $1 \le r \le \frac{n}{2}+1$, we can increment it by $n-1$ with $\pi_r$. For $r=1,2$ and $1 \le s \le \frac{n}{2}-1$, we can increment $r$ by $2s$ with $\pi_{r+s} \pi_r$. (First $\pi_r$ sends $r$ to $r+n-1$. Then observe that $r+n-1$ is $s$ away from $r+s+n-1$, which is the highest entry in $\pi_{r+s}$, and $r+2s$ is $s$ away from $r+s$, which is the lowest entry in $\pi_{r+s}$.) Also, if $r=1$, we can increment it by 1 with $\pi_{r+1} \pi_{r+n/2} \pi_r$.

Hence we can obtain a permutation that sends 1 to 2, or to $n$, or to $2s+1$ for $1 \le s \le \frac{n}{2}-1$. We can also obtain a permutation that sends 1 to $2s+2$, for $1 \le s < \frac{n}{2}-1$, by first sending it to 2 and then incrementing by $2s$. Thus, for any $N \le n$, we can obtain a sequence of $n$-crossing permutations over $S_{3n/2}$ that sends 1 to $N$.

If $n < N \le \frac{3n}{2}$, then we can first take a permutation that sends 1 to some element $N-n+1$, where $1 < N-n+1 \le \frac{n}{2}+1$. Then we can compose it with $\pi_{N-n+1}$, which will send 1 to $N$.
\end{proof}

The next lemma describes how to increment the largest entry in a permutation.

\begin{lemma}
\label{lem:moveMax}
Let $n$ be even. Then, given a permutation of the form
\[
\tau = (1, 2, \dots, l_1)(l_1+1, \dots, l_2) \cdots (l_{t-1}+1, \dots, l_t-1, l_t),
\]
we can conjugate it by the $n$-crossing permutations over $S_{l_t+3n-3}$ to obtain any permutation of the form
\[
(1, 2, \dots, l_1) (l_1+1, \dots, l_2) \cdots (l_{t-1}+1, \dots, l_t-1, N),
\]
where $l_t \le N \le l_t+3n-3$.
\end{lemma}

\begin{proof}
First observe that over $S_{l_t + 3n-3}$, we have the $n$-crossing permutations $\pi_1, \dots, \pi_{l_t + 2n-2}$.

By applying Lemma~\ref{lem:moveMaxPerm} to the strings from $l_t$ to $l_t + \frac{3n}{2}-1$, we can see that for any $N$ with $l_t \le N \le l_t + \frac{3n}{2}-1$, there is a sequence of $n$-crossing permutations over $S_{l_t+3n/2-1}$ which sends $l_t$ to $N$, and leaves $1,\dots,l_t-1$ fixed. Then we can conjugate $\tau$ by this sequence of $n$-crossing permutations to obtain the permutation $(1, \dots, l_1)\cdots (l_{t-1}+1, \dots, l_t-1, N)$.

If $l_t+\frac{3n}{2} \le N \le l_t+3n-3$, then we first obtain a permutation $(1, \dots, l_1)\cdots (l_{t-1}+1, \dots, l_t-1, N')$ where $N'$ is an element with $l_t \le N' \le l_t + \frac{3n}{2}-1$ that is a multiple of $n-1$ away from $N$. Then we can conjugate by $\pi_{N'}$, $\pi_{N'+n-1}$, and so on, until we reach $N$ after conjugating by $\pi_{N-n+1}$. Note that in doing so, we only need the crossing permutations $\pi_1, \dots, \pi_{l_t + 2n-2}$. Thus, the crossing permutations over $S_{l_t+3n-3}$ are sufficient to produce any permutation of the form $(1, \dots, l_1)\cdots (l_{t-1}+1, \dots, l_t-1, N)$, where $N \le l_t+3n-3$.
\end{proof}

The next lemma describes how to increment the remaining entries of a permutation, without changing the order of the entries.

\begin{lemma}
\label{lem:moveAll}
Let $n$ be even. Then, given $\tau$ as in the statement of Lemma~\ref{lem:moveMax}, we can conjugate it by the $n$-crossing permutations over $S_{l_t+3n-3}$ to obtain any permutation of the form
\[
(a_1, \dots, a_{l_1}) (a_{l_1+1}, \dots, a_{l_2}) \cdots (a_{l_{t-1}+1}, \dots, a_{l_t})
\]
where $1 \le a_1 < a_2 < \cdots < a_{l_t} \le l_t+3n-3$.
\end{lemma}

\begin{proof}
We prove this by iterating over each entry in $\tau$, from largest to smallest. The base case (largest entry in $\tau$) can be done by Lemma~\ref{lem:moveMax}.

Suppose we have moved the $k$ largest entries to their appropriate positions, where $1 \le k \le l_t$. Let $N = a_{l_t-k+1}$, which is to say that we have a permutation of the form
\[
\tau' = (1, \dots, l_1) \cdots (l_{s-1}+1, \dots, l_t-k, N, a_{l_t-k+2}, \dots, a_{l_s}) \cdots (a_{l_{t-1}+1}, \dots, a_{l_t}),
\]
Let $N' = a_{l_t - k}$, the new desired entry for the $k+1$st largest entry. Note that $N > N'$ by the hypothesis. We want to move $l_t-k$ to $N'$.

First consider the case when $N \le l_t - k + \frac{3n}{2} - 1$. Note that in this case $N' < l_t - k + \frac{3n}{2} - 1$. Observe that since we have only moved the $k$ largest entries, one of the $k$ indices between $l_t - k + \frac{3n}{2} - 1$ and $l_t + \frac{3n}{2} - 2$ must be ``empty'', which is to say it does not appear in $\tau'$, or that it is a fixed point in $\tau'$. Let $M$ be such an element.

We apply Lemma~\ref{lem:moveMaxPerm} to the $\frac{3n}{2}$ strings starting at $l_t - k + \frac{3n}{2} - 1$. Note that we need $l_t + 3n - 3$ strings for this to be possible for all $k$ with $1 \le k \le l_t$. Then, there exists a sequence $\{\pi_{b_i}\}$ of $n$-crossing permutations over $S_{l_t - k + 3n - 2}$ that sends $l_t - k + \frac{3n}{2} - 1$ to $M$, and leaves $1, \dots, l_t - k + \frac{3n}{2} - 2$ fixed. Note that this sequence in reverse order sends $M$ to $l_t-k+\frac{3n}{2} - 1$. Hence if we conjugate $\tau'$ by $\{\pi_{b_i}\}$ in reverse order, then the element $l_t-k+\frac{3n}{2} - 1$ will not appear in the resulting permutation. Thus we have some permutation of the form
\[
(1, \dots, l_1) \cdots (l_{s-1}+1, \dots, l_t-k, c_{l_t-k+1}, c_{l_t-k+2} \dots, c_{l_s}) \cdots (c_{l_{t-1}+1}, \dots, c_{l_t}),
\]
where all $c_i$ are strictly greater than $l_t - k + \frac{3n}{2} - 1$. By Lemma~\ref{lem:moveMaxPerm}, there is a sequence of $n$-crossing permutations $\{\pi_{b_i}\}$ over $S_{l_t-k+3n/2-1}$ which sends $l_t$ to $N'$, and leaves $1,\dots,l_t-k-1$ fixed. (Note that $\{\pi_{b_i}\}$ also leave entries greater than $l_t-k+\frac{3n}{2}-1$ fixed.) Then we can take conjugates by $\{\pi_{b_i}\}$ to obtain the permutation
\[
(1, \dots, l_1) \cdots (l_{s-1}+1, \dots, l_t-k-1, N', c_{l_t-k+1}, c_{l_t-k+2} \dots, c_{l_s}) \cdots (c_{l_{t-1}+1}, \dots, c_{l_t}).
\]
Finally, we can take conjugates by $\{\pi_{b_i}\}$ in the forward order to move the other entries back and obtain
\[
(1, \dots, l_1) \cdots (l_{s-1}+1, \dots, l_t-k-1, N', N, a_{l_t-k+2} \dots, a_{l_s}) \cdots (a_{l_{t-1}+1}, \dots, a_{l_t}).
\]
Note that since $N' < l_t - k + \frac{3n}{2} - 1$, the conjugations by $\{\pi_{b_i}\}$ does not affect $N'$.

Next, consider the case when $N' \le l_t - k + \frac{3n}{2} - 1 < N$. In this case, by Lemma~\ref{lem:moveMaxPerm} we can move $l_t-k$ to $N'$ without affecting any other elements.

Finally, consider the case when $N' > l_t - k + \frac{3n}{2} - 1$. First we use Lemma~\ref{lem:moveMaxPerm} to move $l_t-k$ to some $M'$, where $M'$ is an element with $l_t - k \le N' \le l_t + \frac{3n}{2}-1$ that is a multiple of $n-1$ away from $N$. Then we can conjugate by $\pi_{M'}$, $\pi_{M'+n-1}$, and so on, until we reach $N'$.
\end{proof}

Finally we prove the desired result. We define the \emph{size} of a permutation to be the number of distinct entries that appear in the permutation when written in cycle notation, where we drop any 1-cycles.

\begin{lemma}
\label{lem:sameCycleType}
Let $n$ be even. Then, given any $\tau \in S_{N+3n-3}$ with size at most $N$, we can conjugate by the $n$-crossing permutations over $S_{N+3n-3}$ to obtain any permutation in $S_{N+3n-3}$ of the same cycle type as $\tau$.
\end{lemma}

\begin{proof}
Write
\[
\tau = (d_1, \dots, d_{l_1})(d_{l_1+1}, \dots, d_{l_2}) \cdots (d_{l_{t-1}+1}, \dots, d_{l_t})
\]
in cycle notation. Note that $l_t \le N$ since the size of $\tau$ is at most $N$. Then consider the permutation
\[
\tau' = (1, \dots, l_1)(l_1+1, \dots, l_2) \cdots (l_{t-1}+1, \dots, l_t).
\]
Note $\tau'$ has the same cycle type as $\tau$. It suffices to show that we can conjugate $\tau'$ to get any permutation of the same cycle type, since we can reverse this process to obtain $\tau'$ from $\tau$.

Since $l_t \le N$, we know we can use the $n$-crossing permutations over $S_{l_t + 3n - 3}$. Thus we can apply Lemma~\ref{lem:moveAll} to obtain any permutation of the form 
\[
(a_1, \dots, a_{l_1})(a_{l_1+1}, \dots, a_{l_2}) \cdots (a_{l_{t-1}+1}, \dots, a_{l_t}),
\]
where $a_1 < \cdots < a_{l_t}$. Hence it suffices to show that we can switch two entries in the permutation. We show how to switch two entries in $\tau'$, after which all entries can be moved to the appropriate positions.

Suppose we want to switch $r$ and $r+1$ for $r \le l_t$; in other words, suppose we want to obtain the permutation
\[
(1, \dots, l_1) \cdots (l_{s-1}+1, \dots, r-1, r+1, r, r+2, \dots, l_s) \cdots (l_{t-1}+1, \dots, l_t).
\]
Note that $s$ and $s+1$ may appear in different cycles in $\tau$, but the same argument applies. First we can conjugate $\tau'$ by $\pi_r$ to obtain the permutation
\[
(1, \dots, l_1) \cdots (l_{s-1}+1, \dots, r-1, r+n-1, r+n-2, e_{r+2} \dots, e_{l_s}) \cdots (e_{l_{t-1}+1}, \dots, e_{l_t}),
\]
where the $e_i$ are strictly less than $r+n-2$ for $i \ge r+2$. Then we can conjugate by $\pi_{r+n-2}$ to obtain the permutation
\[
(1, \dots, l_1) \cdots (l_{s-1}+1, \dots, r-1, r+2n-4, r+2n-3, e_{r+2}, \dots, e_{l_s}) \cdots (e_{l_{t-1}+1}, \dots, e_{l_t}).
\]
We can then conjugate by $\pi_{r+\frac{3n}{2}-3}$ to obtain the permutation
\[
(1, \dots, l_1) \cdots (l_{s-1}+1, \dots, r-1, r+2n-3, r+2n-4, e_{r+2}, \dots, e_{l_s}) \cdots (e_{l_{t-1}+1}, \dots, e_{l_t}).
\]
Observe that we have switched the $r$th entry with the $r+1$st entry. Finally, we can conjugate by $\pi_{r+n-2}$ and then by $\pi_r$ to move all entries back and obtain the desired permutation.
\end{proof}

While this suffices to prove Theorem~\ref{thm:evenAlexander}, we wish to reduce the number $d+3n-3$ as much as possible. As we will see in Section~\ref{sec:braidIndices}, we conjecture that it can be reduced down to $n+2$ for $d \le n$.

\subsection{Proof of Theorem~\ref{thm:evenAlexander}}

Before proving Theorem~\ref{thm:evenAlexander}, we present one final lemma.

\begin{lemma}
\label{lem:evenPair}
Let $n$ be even, and let $m \ge \frac{5n}{2}-1$. Then we can obtain the permutations $(1,n)(2n-1,2n)$ as a product of the $n$-crossing permutations over $S_m$.
\end{lemma}

\begin{proof}
First note that over $S_{5n/2-1}$ we have the crossing permutations $\pi_1, \dots, \pi_{3n/2}$.

Take the crossing permutation $\pi_1 = (1,n)(2,n-1)\cdots(\frac{n}{2}, \frac{n}{2}+1)$. We can conjugate by $\pi_n$ to obtain the permutation $(1,2n-1)(2,n-1)\cdots(\frac{n}{2}, \frac{n}{2}+1)$. Then conjugate by $\pi_{3n/2}$ to get $(1,2n)(2,n-1)\cdots(\frac{n}{2}, \frac{n}{2}+1)$. Thus we have moved the largest entry of the permutation to $2n$.

We then perform a similar sequence to move the smallest entry to $2n-1$. This is done by conjugating by $\pi_1$ and then by $\pi_n$. We are now left with the permutation $(2n-1,2n)(2,n-1)\cdots(\frac{n}{2},\frac{n}{2}+1)$. Note that the first conjugation by $\pi_1$ also moves the other entries around, but the pairs stay the same; for example $\pi_1$ switches the entries 2 and $n-1$, but this keeps the cycle $(2,n-1)$ constant.

Finally we can multiply this permutation by $\pi_1$. The transpositions in the middle would cancel, and we are left with $(1,n)(2n-1,2n)$.
\end{proof}

We are now ready to prove Theorem~\ref{thm:evenAlexander}.

\begin{proof}[Proof of Theorem~\ref{thm:evenAlexander}]
Let $L$ be a link. Put it in a double crossing braid form; call this braid $\alpha$.

First consider the case when $n = 4k + 2$. Start with $\alpha$, and take stabilizations until we have at least $3n+1$ strands; call this braid $\alpha'$.

By Lemma~\ref{lem:evenPair}, we can obtain the permutation $(1,n)(2n-1,2n)$. This permutation has size 4, so by Lemma~\ref{lem:sameCycleType}, we can conjugate $\pi_1$ to obtain any permutation whose cycle type is two transpositions. 
We start with $\pi_1$, which has an odd number of transpositions. We can cancel pairs of transpositions if we multiply by the pairs, which we can obtain by Lemma~\ref{lem:sameCycleType}. Thus we can cancel all but one of the transpositions. By Lemma~\ref{lem:sameCycleType}, we can rearrange the entries of this transposition to obtain any transposition of the form $(i,i+1)$ for any $i$ with $1 \le i < 3n+1$. Hence, by Lemma~\ref{lem:levelPosition}, we can produce any double crossing that appears in $\alpha'$ as a product of $n$-crossings. Thus, we have produced an $n$-crossing braid that is equivalent to $L$.

Now consider the case when $n=4k$. Observe that the permutations obtained from the $n$-crossings are all even. Therefore these permutations cannot generate the transpositions, which are odd. However, we can start with a double crossing braid with an even number of crossings, which will be generated by all disjoint pairs of double crossings. Note that if two consecutive crossings involve a common strand, we can obtain one of them with a pair of crossings by creating an extra dummy crossing elsewhere, and obtain the other with a pair which undoes the dummy crossing.

Start with $\alpha$, and take stabilizations until we have at least $3n+1$ strands. If this braid has an odd number of crossings, take an extra stabilization so we have an even number of crossings. From the top of the braid, put the crossings in disjoint pairs, adding a pair of dummy crossings as necessary. Call this braid $\alpha'$.

As before, by Lemma~\ref{lem:evenPair} and Lemma~\ref{lem:sameCycleType} we can obtain any permutation whose cycle type is two transpositions. We start with $\pi_1$, which has an even number of transpositions. We can cancel pairs of transpositions as before, so we are left with a pair of transpositions. Then, by Lemma~\ref{lem:sameCycleType} we can rearrange the entries to get any pair of transpositions of the form $(a, a+1)(b,b+1)$, where the four entries are all distinct. Hence, by Lemma~\ref{lem:levelPosition}, we can produce any disjoint pair of double crossings that appears in $\alpha'$ as a product of $n$-crossings. Thus, we have produced an $n$-crossing braid that is equivalent to $L$.
\end{proof}

\section{Triple-crossing Braids}
\label{sec:tripleAlexander}

Once we know that every link can be put into an $n$-crossing braid for even $n$, a natural question to ask is: is this also true when $n$ is odd?

First, consider an open braid with $m$ strings. If we label the strings $1,\dots,m$ from left to right, we can see that any odd multi-crossing only mixes strings with the same parity. The closure of the braid must therefore have at least two components, which means that we cannot put knots in an odd multi-crossing braid form. However, we claim that we can put any link with two or more components into any odd multi-crossing braid form. In this section we prove this for $n=3$.

\begin{theorem}
\label{thm:tripleAlexander}
Every link with two or more components can be represented as a closed triple crossing braid. 
\end{theorem}

This will be extended to other odd $n$ in the next section.

\subsection{Setup}

Consider an open braid with $m \ge 3$ strings. As before, we can label the strings $1,\dots,m$ from left to right, at the top of the whole braid. We call this label the \emph{index} of a string. At a given section of the braid, we can also label the strings $1, \dots, m$ from left to right at the top of this section. We call this the \emph{position} of the string in the section. Note that for a given string, the index stays the same from the top of the braid to the bottom, but the position changes every time it is involved in a crossing.

First we show a lemma which takes a double crossing braid and isotopes it into a form that is easier to work with.

\begin{lemma}
\label{lem:evenOdd}
Let $\alpha$ be a closed double crossing braid with at least 2 components. Then we can find an isotopy of $\alpha$ such that in the resulting braid, a set of components always enter and leave the braid in an odd position, while the remaining components always enter and leave the braid in an even position.
\end{lemma}

\begin{proof}
First, look at the the string with index 1, and remember this to be an odd component. Then we check if the string with index 2 is another component, which we remember to be an even component. If it is the same component, then we can find some string from another component that has a larger index, and make it have index 2 as follows. If the string from another component has index $i$, then we can conjugate the open braid by $\sigma^{i-1}\sigma^{i-2}\cdots\sigma^2$, so that this string now has index 2.

Then we continue this process, checking at every step that a string with an odd (similarly even) index $i$ is an odd (similarly even) component or a new component, and if not, finding an appropriate component and making it index $i$. Then at the top of the braid, the strings alternate between the odd and the even components. If we do not have enough strands for the odd components or the even components, then we can perform stabilizations on a component to obtain new strings.
\end{proof}

Note that once we go through this process and make sure that the top of the braid alternates between odd and even components, we know that the bottom of the braid follows the same pattern. This is because if a string leaves the braid in an odd position $i$, it must be the same component as the string that enters the braid at the $i$th index. This means that this string must be an odd component. Similarly a string that leaves the braid in an even position must be an even component. This also means that if a string enters the braid at an odd (similarly even) index, then it must leave the braid at an odd (similarly even) index. In other words, the parity of the position of the string at the bottom of the braid must be the same as the parity of its index.

\subsection{Putting the braid in level position}

Recall from Section~\ref{sec:evenAlexander} that we can put the braid in level position; we assign heights to each string, and if we put in crossings in a way that keeps each string on its level, we can pull the strings taut and then we are left with a set of crossings which represents the corresponding permutation of the strings. We can assign these heights according to the index of each string, so that the index and the level coincide.

We describe a process to isotope this braid so that, starting at the top of the braid and moving down, we are left with a sequence of triple crossings and then a double crossing braid in level position. (Note that when we refer to the ``top'' or the ``bottom'' of the braid, we always mean the beginning and the end of the braid word, rather than the heights that are assigned to a braid in level position, which we refer to in terms of its level.)

First we define a notion that becomes key to this argument. A \emph{clasp} is two strings that are hooked together as below (Fig.~\ref{fig:clasp}).

\begin{figure}[ht]
	\centering
	\includegraphics[scale=0.4]{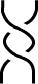}
	\caption{Two strings hooked together is called a clasp.}
	\label{fig:clasp}
\end{figure}

Note that in a braid, a clasp appears as $\sigma_i \sigma_i$.

\begin{lemma}
\label{lem:tripleLevel}
Let $\alpha$ be a double crossing braid with 3 or more strings. We can find an isotopy of this braid such that we are left with a sequence of triple crossings followed by a double crossing braid in level position.
\end{lemma}

\begin{proof}
By Lemma~\ref{lem:evenOdd}, we can find an isotopy of $\alpha$ such that the parity of the position of each string is preserved from the top of the braid to the bottom.

In the course of this argument, we can ``ignore'' any sequence of triple crossings at the top of the braid, as long as we do not try to take conjugations or stabilizations. This is because triple crossings always preserve the parity of the position of each string, which is the property that is central to this argument. If we can turn some double crossing braid into a triple crossing braid, we can clearly do the same to a sequence of triple crossing braid followed by this double crossing braid.

Start with a double crossing braid $\alpha$ such that the parity of the position of each component remain the same. Assign the indices $1, \dots, m$ to the $m$ strings. We can then find a new (different) double crossing braid $\alpha'$, with the same projection as $\alpha$ but different crossings, that is in level position with respect to this leveling. We find an isotopy of $\alpha$ such that we are left with a sequence of triple crossings followed by $\alpha'$.

We start from the top of the braid $\alpha$, and at every step we find a crossing that is different from the braid $\alpha'$ in level position, and change this crossing so that we are closer to being in level position. Note that in changing this crossing, we may introduce triple crossings at the top of the braid to preserve the original braid type, but we are not concerned with this.

This will be done recursively. Suppose we want to change a double crossing at the top of the braid. Then we can introduce two trivial crossings under this crossing. Then the two crossings at the top can be turned into triple crossings as below. Thus we obtain a double crossing braid with this top crossing flipped (Fig.~\ref{fig:tripleTop}).

\begin{figure}[ht]
	\centering
	\includegraphics[scale=0.3]{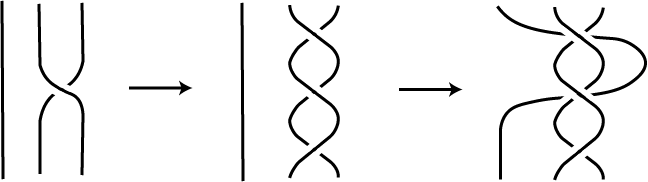}
	\caption{Changing a crossing at the top of the braid. Introduce two trivial crossings below, and use an extra strand to turn the two top crossings into triple crossings. We are now left with two triple crossings followed by a double crossing, which is now different from the one we started with.}
	\label{fig:tripleTop}
\end{figure}

Now suppose we want to change a crossing $\sigma$ in the middle of a braid, assuming that all crossings above it have been changed so that the portion above it is in level position. Let $A$ be the portion of the braid above this crossing, not including itself. Then we can add trivial crossings and obtain $A\sigma = A\sigma \sigma A^{-1} A \sigma^{-1}$ (Fig.~\ref{fig:tripleSequence}). It suffices to turn $A \sigma \sigma A^{-1}$ into triple crossings, for this would mean we have changed $A \sigma$ into a sequence of triple crossings followed by $A \sigma^{-1}$. Note that since $A$ is in level position, we know that $A^{-1}$ must also be in level position with the same leveling of strings.

\begin{figure}[ht]
	\centering
	\includegraphics[scale=0.25]{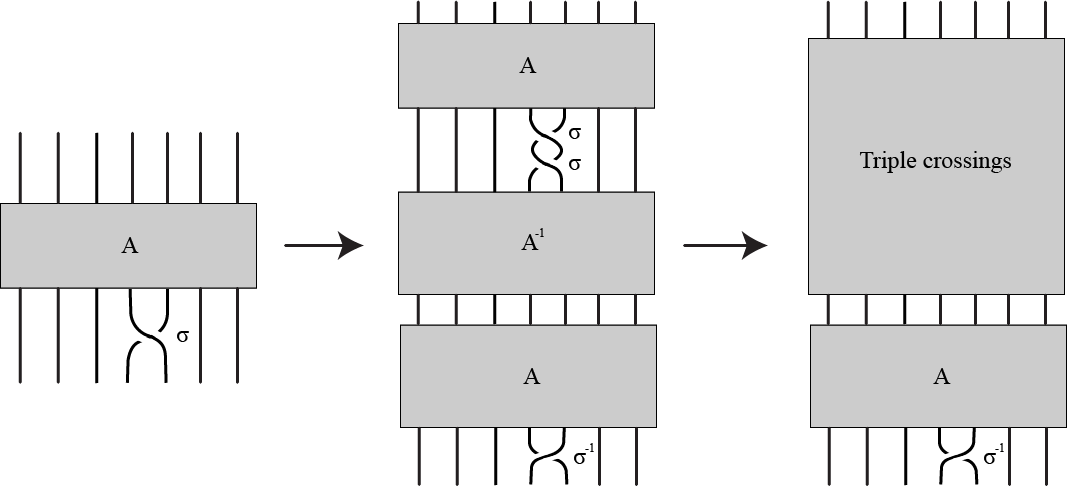}
	\caption{Starting with $A \sigma$, we can add $A^{-1} A \sigma^{-1}$ below it, since it is equal to the identity. If we can turn $A \sigma \sigma A^{-1}$ into triple crossings, then we are left with $A \sigma^{-1}$ as desired.}
	\label{fig:tripleSequence}
\end{figure}

Let $\sigma$ be a crossing that switches the strings with indices $a,b$, where $a < b$. Note that the $\emph{position}$ of these strings must be adjacent on this portion of the braid. Recall that $A$ and $A^{-1}$ are in level position, so these strings are in the levels $a$, $b$ respectively in both $A$ and $A^{-1}$. The crossings $\sigma \sigma$ form a clasp that violates the level position, and may get in the way of strings that are between $a$ and $b$. For clarification, call the first crossing of the clasp $\sigma_r$, and the second one $\sigma_s$.

Let $i$, $i+1$ be the positions of the two strings at the top of the clasp. Note that because $\sigma_r$ is a crossing that violates the level position, $a$ must be the understrand of this crossing. Then $a$ is the overstrand of $\sigma_s$.

We consider two possible cases: $a$ could be in the $i$th position at the top of the clasp, or it could be in the $i+1$st position. If $a$ is in $i+1$st position, this means that $b$ is in the $i$th position. Now, pull the parts of $a$ and $b$ in $A$ and $A^{-1}$ taut. Then there must be a crossing between $a$ and $b$ in both $A$ and $A^{-1}$. Let $\sigma_p$ the one in $A$, and let $\sigma_q$ be the one in $A^{-1}$. Observe that $a$ is the overstrand of both of these crossings, since $a$ is in a higher level than $b$. Then $\sigma_q$ cancels with $\sigma_s$, since $a$ is the overstrand in both. We can then move $\sigma_p$ down so that $\sigma_p$ and $\sigma_r$ is a clasp between the $a$ and $b$ strings; however $a$ is now in the $i$th position at the top of the clasp, and $b$ is in the $i+1$st (Fig.~\ref{fig:tripleSigma}). Note that the portions above and below this clasp are still in level position. Hence we can assume that $a$ is the $i$th position at the top of the clasp, and $b$ is in the $i+1$st position.

\begin{figure}[ht]
	\centering
	\includegraphics[scale=0.25]{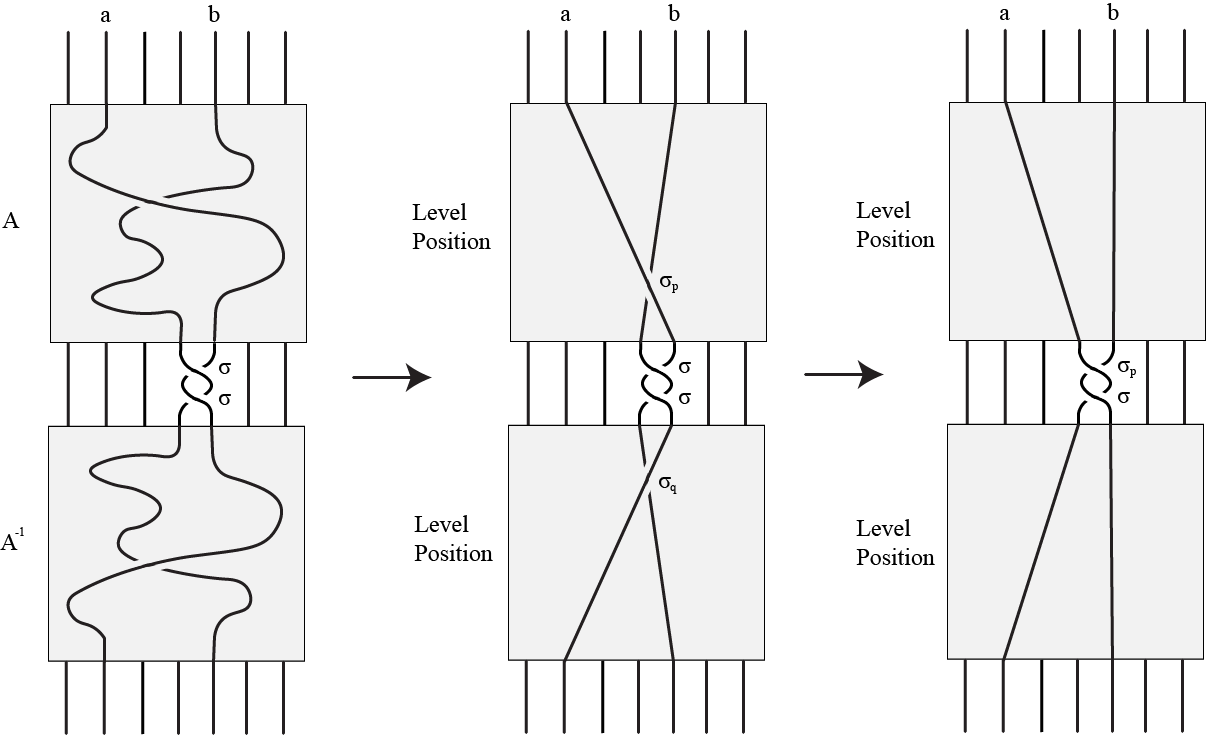}
	\caption{If $a$ is the $i+1$st position at the top of the clasp, then there must be a crossing $\sigma_p$ between $a$ and $b$ in $A$, and a corresponding crossing $\sigma_q$ in $A^{-1}$. Then $\sigma_q$ cancels with $\sigma_s$, and a clasp is formed by $\sigma_p$ and $\sigma_r$. The portions above and below the clasp are still in level position. Thus we can assume that $a$ is in the $i$th position at the top of the clasp.}
	\label{fig:tripleSigma}
\end{figure}

The strings $a$ and $b$ look like the following figure (Fig.~\ref{fig:tripleDiagonalMess}). Both strings stay on their levels until they reach the clasp, where they wrap around each other and move back to their original levels.

\begin{figure}[ht]
	\centering
	\includegraphics[scale=0.2]{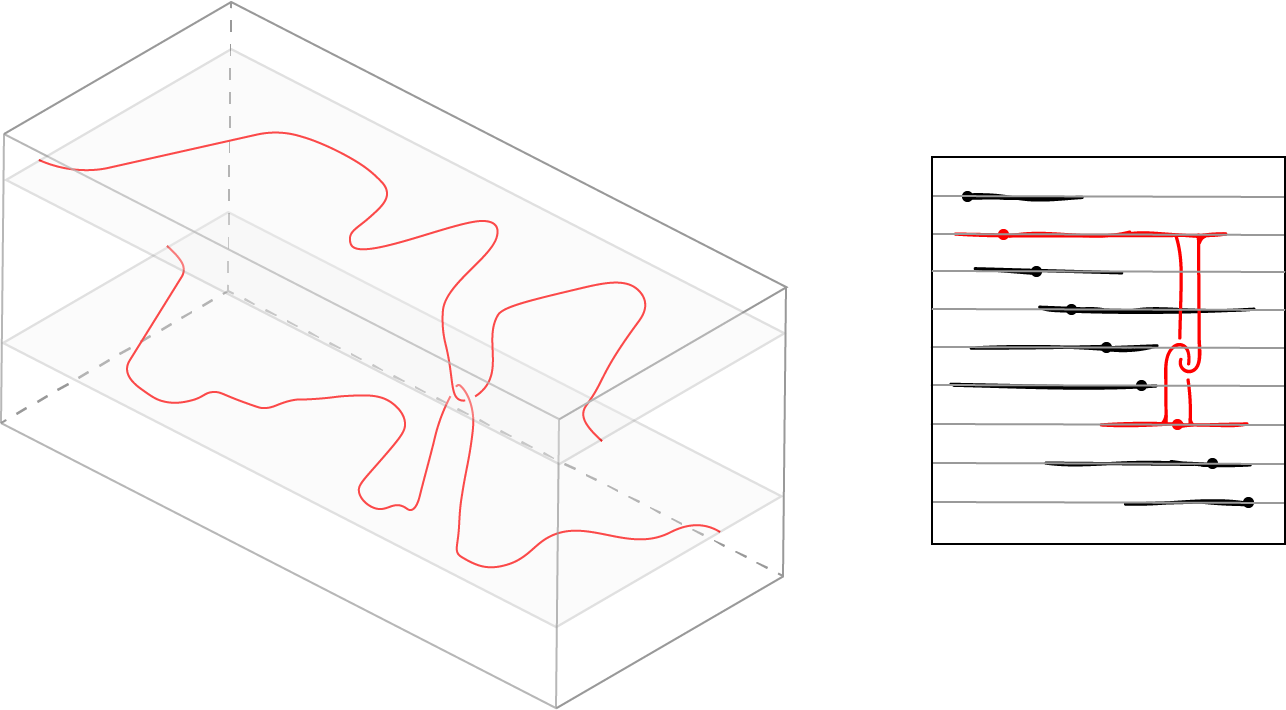}
	\caption{The left picture depicts the strings $a$ and $b$. Note that the other strings are not drawn, but they would all stay on the same level. The right picture is a view from the side, where this time the other strings are also drawn. The two strings move toward the level of the other string, wrap around each other, and then go back to their original levels.}
	\label{fig:tripleDiagonalMess}
\end{figure}

Now pull all the strings taut. All strings on levels higher than $a$ or lower than $b$ (note lower numbers are on higher levels here) are not affected, so they will go straight down from the top of the braid to the bottom. The strings $a$ and $b$ create a diagonal plane between the $a$th and the $b$th levels. The strings that are in between the levels $a$ and $b$ will be either above or below this plane, depending on what its position was at the clasp. If its position was greater than $i+1$, then it will be above this diagonal plane; it its position was less than $i$, then it will be below the plane. Thus we have the following picture (Fig.~\ref{fig:tripleDiagonalTaut}).

\begin{figure}[ht]
	\centering
	\includegraphics[scale=0.2]{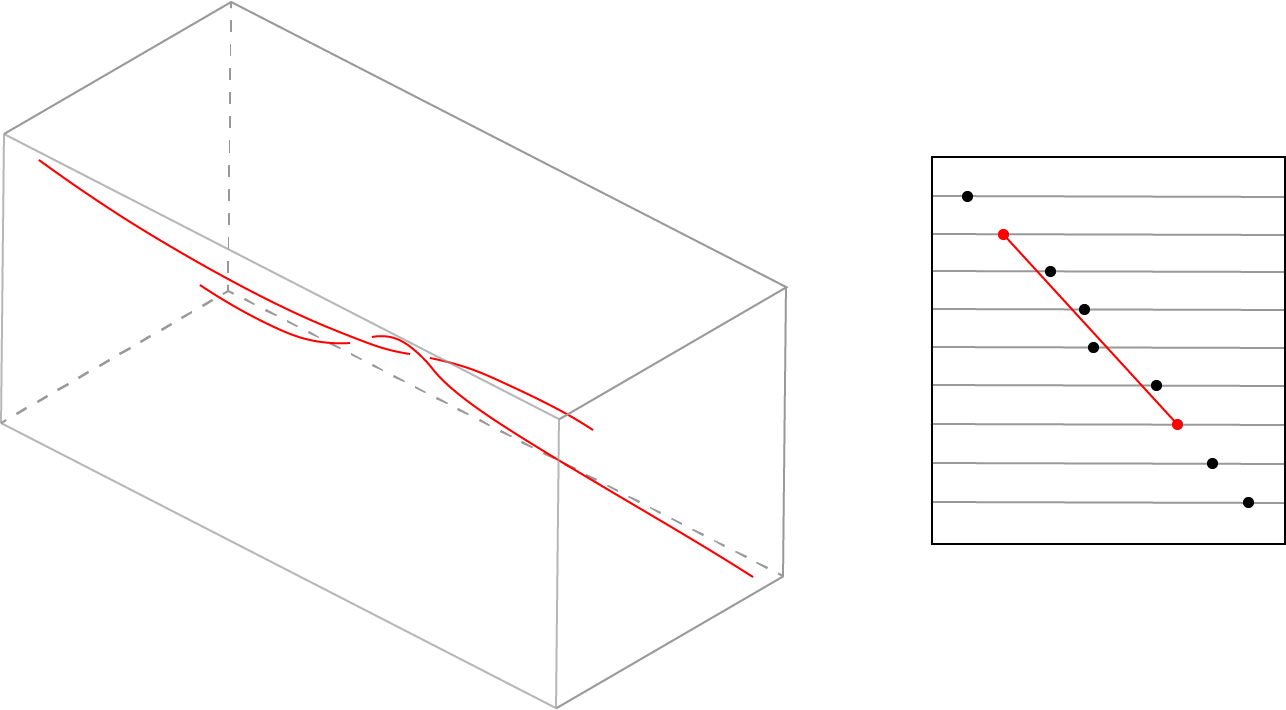}
	\caption{The left picture shows what happens to $a$ and $b$ once the strings are pulled taut. The right picture shows the view from the side, where we see which strings lie above/below the diagonal plane created by $a$ and $b$.}
	\label{fig:tripleDiagonalTaut}
\end{figure}

Therefore when all of these strings are pulled taut we have the following picture (Fig.~\ref{fig:tripleHook}).

\begin{figure}[ht]
	\centering
	\includegraphics[scale=0.3]{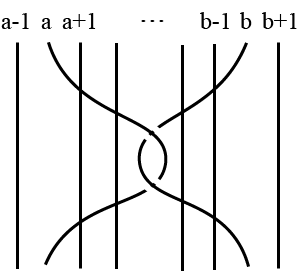}
	\caption{Once the strings are pulled taut, we are left with one clasp between the strings $a$ and $b$, and all other strings are either above or below this clasp.}
	\label{fig:tripleHook}
\end{figure}

The strings that are in between $a$ and $b$ will go over both strands of the clasp, or under both strands of the clasp. This means that all of these $b-a-1$ strings can be moved to one side of the clasp. But any pair of such strings can be turned into two triple crossings as below (Fig.~\ref{fig:tripleResolve}). 

\begin{figure}[ht]
	\centering
	\includegraphics[scale=0.3]{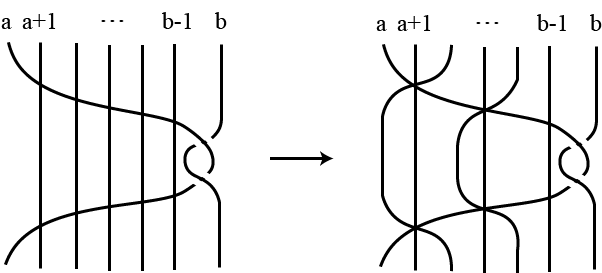}
	\caption{We can move the strings between $a$ and $b$ to one side, and resolve pairs of them into triple crossings.}
	\label{fig:tripleResolve}
\end{figure}

Therefore we only need to consider the case when one string and the clasp are left, or the case when only the clasp is left. If we have one string left we can simply let it run through the middle of the clasp (Fig.~\ref{fig:tripleResolveLast}). Note that we can always do this because it can be under or over both strings of the clasp, but it is never over one and under the other. If only the clasp is left we can take a strand from another string, and pull it under both crossings of the clasp. We can always do this because we have assumed that we have at least 3 strings.

\begin{figure}[ht]
	\centering
	\includegraphics[scale=0.3]{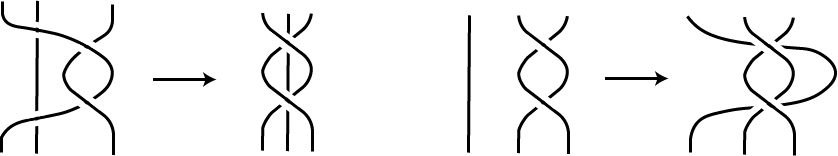}
	\caption{If one string is left, we let it run through the middle of the clasp. The left figure depicts the case when this string goes under both strands of the clasp. If only the clasp is left, then we can take an extra string and pull it under both crossings of the clasp.}
	\label{fig:tripleResolveLast}
\end{figure}

Therefore we can recursively change every crossing to turn the braid into a sequence of triple crossings followed by a double crossing braid in level position.
\end{proof}

\subsection{Obtaining the triple crossing braid}

All we have to do now is to turn this double crossing braid in level position into a triple crossing braid.

\begin{proof}[Proof of Theorem~\ref{thm:tripleAlexander}]
Let $L$ be a link with at least 2 components. Put it in double crossing braid form. If it is a 2-braid, take a stabilization so that it has at least 3 strings. By Lemma~\ref{lem:tripleLevel}, we can find an isotopy of this braid so that it becomes a sequence of triple crossings followed by a double crossing braid in level position. It remains to isotope this double crossing braid into a triple crossing braid.

Recall that each string enters and leaves the braid in positions with the same parity. We can pull the strings taut so that the braid is just a permutation of the even components, and a permutation of the odd components. We can then shift the strings slightly so that the permutation of the odd components occur in the top half of the braid, and the permutation of the even components occur in the bottom half of the braid (Fig.~\ref{fig:tripleLevel}).

\begin{figure}[ht]
	\centering
	\includegraphics[scale=0.3]{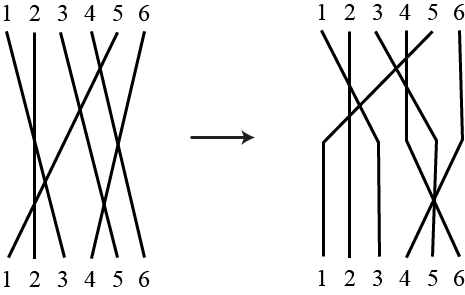}
	\caption{Once the strings are pulled taut, we are left with a permutation of the odd strings and a permutation of the even strings. Since the braid is in level position, we can separate the two permutations as above.}
	\label{fig:tripleLevel}
\end{figure}

Now, any permutation of the odd strands is generated by a transposition of two adjacent odd strands. Note that a triple crossing switches two adjacent odd strands. We can therefore obtain any permutation of the odd strands as a sequence of triple crossings, and similarly for a permutation of the even strands. Thus we have a triple crossing braid form for $L$.
\end{proof}

\section{Odd Multi-crossing Braids}
\label{sec:oddAlexander}

We have already seen that given a link $L$ with two or more components, it can be put into a triple-crossing braid form. In this section, we extend this result to any $n$-crossings, where $n$ is odd. Our goal is to prove the following theorem:

\begin{theorem}
\label{thm:oddAlexander}
Every link with two or more components can be represented as a closed $n$-crossing braid, for all $n$.
\end{theorem}

\subsection{Setup}

We start by putting $L$ in a triple crossing braid form by Theorem~\ref{thm:tripleAlexander}, and finding an equivalent $n$-crossing braid. In other words, we want to show that we can produce the triple crossings from a sequence of $n$-crossings.

As in the case for $n$ even, by Lemma~\ref{lem:levelPosition} it suffices to show that we can obtain each triple crossing permutation as a combination of the $n$-crossing permutations. The triple crossing permutations are all transpositions of the form $(i, i+2)$, and when $n$ is odd, the $n$-crossing permutations are $\pi_j = (j,j+n-1)(j+1,j+n-2)\cdots(j+\frac{n-3}{2},j+\frac{n+1}{2})$. Observe that there are two types of triple crossing permutations: those that switch even numbered strings, and those that switch odd numbered strings. We call them \emph{even-string transpositions} and \emph{odd-string transpositions} respectively.

\subsection{Obtaining pairs of transpositions with same parity}

First we show that given a sufficiently large number of strings, we can use the $n$-crossing permutations to obtain any pair of even-string transpositions, or any pair of odd-string transpositions. Note that these pairs of transpositions must be ``disjoint'' in that the two transpositions permute 4 distinct numbers, for otherwise we would have a 3-cycle or the identity function. While we refrain from using the word ``disjoint'' in this context, it is worth noting that this is different from the meaning of disjoint crossings defined for Lemma~\ref{lem:levelPosition}. For example, a triple crossing permuting the first three strands would have the corresponding transposition $(1,3)$, and a triple crossing permuting the second to fourth strands would have the corresponding transposition $(2,4)$. These two crossings are not disjoint, but the two corresponding permutations are considered to be ``disjoint'' since they involve four distinct numbers.

We can always make sure that we have enough strands as follows. Take the second string from the right, and perform an Type II Reidemeister move over the rightmost string. Then we can stabilize the portion that has now become the rightmost string. We have now added three double crossings, which can be put together into a triple-crossing (Fig.~\ref{fig:3-stabilization}). We call this move a \emph{3-stabilization}.

\begin{figure}[ht]
	\centering
	\includegraphics[scale=0.2]{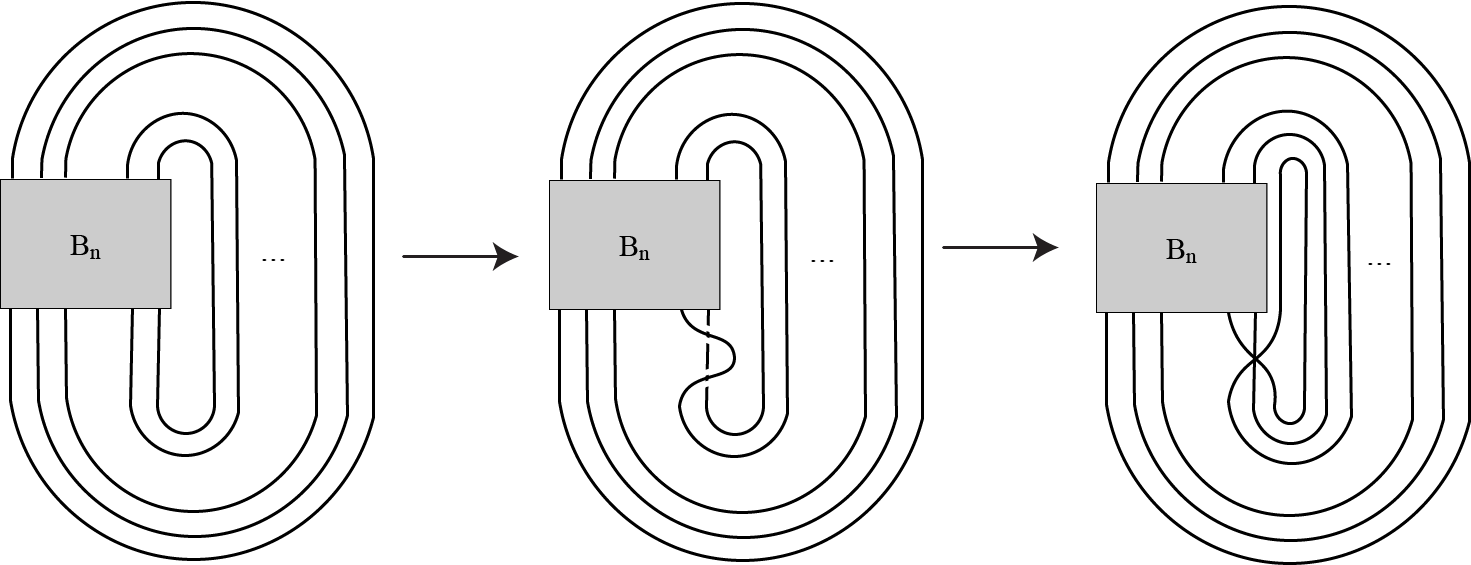}
	\caption{A 3-stabilization.}
	\label{fig:3-stabilization}
\end{figure}

We start by obtaining one pair of even-string transpositions, and one pair of odd string transpositions.

\begin{lemma}
\label{lem:sameParityPairs}
Let $n$ be odd, and let $m \ge \frac{5n-1}{2}$. Then we can obtain the permutations $(3,n)(2n-1,2n+1)$ and $(2,n+1)(2n,2n+2)$, as a product of the $n$-crossing permutations over $S_m$.
\end{lemma}

\begin{proof}
First note that over $S_{(5n-1)/2}$ we have the crossing permutations $\pi_1, \dots, \pi_{(3n+1)/2}$.

Recall that multiplying both sides of $\pi$ by $\sigma$, which is equivalent to conjugating by $\sigma$ since $\sigma^{-1} = \sigma$, switches the entries of $\pi$ according to $\sigma$. Recall also that this does not change the cycle type, which is to say it keeps the number of cycles and the length of each cycle constant. We start with some crossing permutation $\pi_i$ and take conjugates by some $\pi_j$ to obtain different permutations with the same cycle type.

We start by obtaining one pair of odd-string transpositions. Take the crossing permutation $\pi_1 = (1,n)(2,n-1)\cdots(\frac{n-1}{2},\frac{n+3}{2})$. Note that in an odd multi-crossing, the central string is fixed; in this case $\pi_1$ keeps the $\frac{n+1}{2}$ string in the same position. We can conjugate by $\pi_n$ to obtain the permutation $(1,2n-1)(2,n-1)\cdots(\frac{n-1}{2},\frac{n+3}{2})$. Then conjugate by $\pi_{\frac{3n+1}{2}}$ to get $(1,2n+1)(2,n-1)\cdots(\frac{n-1}{2},\frac{n+3}{2})$. Thus we have moved the largest entry of the permutation to $2n+1$.

We then perform a similar sequence to move the smallest entry to $2n-1$. This is done by conjugating by $\pi_1$ and then by $\pi_n$. We are now left with the permutation $(2n-1,2n+1)(2,n-1)\cdots(\frac{n-1}{2},\frac{n+3}{2})$. Note that the first conjugation by $\pi_1$ also moves the other entries around, but the pairs stay the same; for example $\pi_1$ switches the entries 2 and $n-1$, but this keeps the cycle $(2,n-1)$ constant. We can then multiply this permutation by $\pi_1$. The transpositions in the middle would cancel, and we are left with $(1,n)(2n-1,2n+1)$.

Now, we can similarly obtain the permutation $(2,n+1)(2n,2n+2)$, a pair of even-string transpositions, by starting with $\pi_2$.

Finally, we can conjugate $(1,n)(2n-1,2n+1)$ by $\pi_{(n-1)/2}$ to obtain $(1,n-2)(2n-1,2n+1)$. We can conjugate the result by $\pi_1$ to obtain $(3,n)(2n-1,2n+1)$.
\end{proof}

Now we show that we can obtain any pair of even-string transpositions, and any pair of odd string transpositions. We have two cases: $n=4k+3$ and $n=4k+1$.

\begin{lemma}
\label{lem:sameParity4k+3}
Let $n = 4k+3$, and let $m \ge 3n+5$. Then we can obtain any pair of even-string transpositions in $S_m$, or any pair of odd-string transpositions in $S_m$, as a product of the $n$-crossing permutations over $S_m$.
\end{lemma}

\begin{proof}
First observe that $3n+5 = 3(4k+3)+5 = 12k+14 = 2(6k+7)$, so we have at least $6k+7$ odd and even strings respectively.

By Lemma~\ref{lem:sameParityPairs}, we can obtain the permutations $(3,n)(2n-1,2n+1)$ and $(2,n+1)(2n,2n+2)$ as a product of the $n$-crossing permutations over $S_m$. We want to move the entries of this permutation around so that we can get any pair of odd-string transpositions, or any pair of even-string transpositions. 

If we focus on the odd strings and ignore the even strings, we can see that the $(4k+3)$-crossing permutations starting on odd indices function as $(2k+2)$-crossing permutations on the odd strings. Since $2k+2$ is even, we know by Lemma~\ref{lem:sameCycleType} that if we have $3(2k+2)+1 = 6k+7$ odd strings, then we can obtain any odd-string permutation of the same cycle type by taking conjugations. Therefore we can obtain any pair of odd-string transpositions. Similarly we can obtain any pair of even-string transpositions.
\end{proof}

\begin{lemma}
\label{lem:sameParity4k+1}
Let $n = 4k+1$, and let $m \ge 3n-2$. Then we can obtain any pair of even-string transpositions in $S_m$, or any pair of odd-string transpositions in $S_m$, as a product of the $n$-crossing permutations over $S_m$.
\end{lemma}

\begin{proof}
First observe that $3n-2 = 3(4k+1)-2 = 12k+4$ strings. This means that we have $12k+2 = 2(6k+1)$ strings that are not the 1st or the $m$th string. Hence we have $6k+1$ odd strings that are not on either end of the braid, and $6k+1$ even strings that are not on either end of the braid.

As in the proof of Lemma~\ref{lem:sameParity4k+3}, we start with some permutation obtained through Lemma~\ref{lem:sameParityPairs}, and move the entries of these permutations around.

Observe that if we consider an $n$-crossing permutation $\pi_i$ which starts on the $i$th strand, the $(4k+1)$-crossing permutations function as $2k$-crossing permutations on the string with parity different from $i$. Note, however, that none of the $(4k+1)$-crossings can function as a $2k$-crossings that acts on the 1st string or the $m$th string.

First suppose we want to obtain the permutation $(a,b)(c,d)$, where the entries are all odd, and none of them are equal to 1 or $m$. We know by Lemma~\ref{lem:sameParityPairs} that we can obtain the permutation $(3,n)(2n-1,2n+1)$. Since $2k$ is even, we know by Lemma~\ref{lem:sameCycleType} that if we have $3(2k)+1 = 6k+1$ odd strings excluding the first and the last strings, then we can take conjugations and obtain $(a,b)(c,d)$.

Now, suppose we want to obtain $(1,b)(c,d)$, where $b,c,d$ are all odd, and none of them are equal to $m$. Then, using Lemma~\ref{lem:sameParityPairs} and Lemma~\ref{lem:sameCycleType} we can first obtain some $(n,b')(c',d')$ where none of the entries are equal to 1 or $m$. Then we conjugate by $\pi_1$ to obtain $(1,b'')(c'',d'')$. Finally, using the same argument as in the proof of Lemma~\ref{lem:sameCycleType}, but only on the odd strings excluding the first and the last string, we can rearrange the remaining entries to obtain $(1,b)(c,d)$.

By symmetry, we can similarly obtain any permutation with an entry equal to $m$ and none equal to 1; we first obtain $(m-n+1,b')(c',d')$, conjugate by $\pi_{m-n+1}$, and then rearrange the remaining entries. If an entry is equal to 1 and another is equal to $m$, then we can first obtain some permutations where the corresponding entries are equal to $n$ and $m-n+1$ respectively, then conjugate by $\pi_1$ and $\pi_{m-n+1}$, and rearrange the remaining two entries.

We can apply the same argument to obtain any pair of even-string transpositions, by starting with $(2,n+1,2n,2n+2)$. We can use the same argument for the case when one of the entries are equal to $m$.
\end{proof}

Now we present a couple of lemmas for when $n=8k+5$ which will be useful in the proof of Theorem~\ref{thm:oddAlexander}.

\begin{lemma}
\label{lem:triple8k+5}
Let $n=8k+5$. Then we can obtain any odd-string triple crossing permutation in $S_{(3n-1)/2}$, by conjugating $(1,3)$ with the $n$-crossing permutations over $S_{(3n-1)/2}$.
\end{lemma}

\begin{proof}
First observe that over $S_{(3n-1)/2}$, we have the $n$-crossing permutations $\pi_1, \dots, \pi_{(n+1)/2}$.

Suppose we want to obtain the transposition $(2i-1,2i+1)$. First consider the case when $2i-1 \le n$. We conjugate $(1,3)$ by $\pi_1$ to obtain $(n-2,n)$. We then conjugate by $\pi_{(n+1)/2}$ to obtain $(n+2,n)$. We then conjugate by $\pi_{i+1}$ to obtain $(2i-1,2i+1)$. (To see this, observe that $n$ is $j$ away from $i+n$, the highest entry in $\pi_{i+1}$, and that $2j+1$ is $j$ away from $j+1$, the lowest entry in $\pi_{i+1}$.)

Now consider the case when $n < 2i-1 < \frac{3n-2}{2}$. Then we can move $(1,3)$ to some $(j,j+2)$ such that $j \le n$ and $j$ is $n-1$ away from $2i-1$. We can then conjugate $(j,j+2)$ by $\pi_{j+2}$ and then by $\pi_j$ to obtain $(2i-1,2i+1)$.
\end{proof}

\begin{lemma}
\label{lem:pair8k+5}
Let $n=8k+5$, and let $m \ge 4n-4$. Then we can obtain any pair of an even-string transposition and an odd-string transposition over $S_m$, corresponding to a disjoint pair of triple crossings, as a product of the $n$-crossing permutations over $S_m$.
\end{lemma}

\begin{proof}
Observe that when $n=8k+5$, an $n$-crossing permutation consists of $2k+1$ even-string transpositions and $2k+1$ odd-string transpositions. We now show that we can obtain any pair consisting of one odd-string transposition and one even-string transposition whose corresponding crossings are disjoint.

Let $(l_a, l_a+2)(l_b, l_b+2)$ be the permutation we want to obtain, with $l_a+2 < l_b$. First suppose $l_a$ is odd. We can then start $\pi_{(n+1)/2}$, and cancel all but one of the even-string transpositions by multiplying it with pairs of even-string transpositions as before. We can similarly cancel all but one of the odd-string transpositions. Thus we are left with a single even-string transposition and a single odd-string transposition. We may choose the pairs to cancel from the outside so that we are left with a permutation $(n-2,n+2)(n-1,n+1)$.

We first ``separate'' the two transpositions. We can conjugate $(n-2,n+2)(n-1,n+1)$ by $\pi_{n-1}$ to obtain $(n-2,2n-5)(2n-2,2n-4)$. When $n=5$ this is equivalent to $(n-2,n)(n+3,n+1)$.

When $n>5$, we conjugate $(n-2,2n-5)(2n-2,2n-4)$ by $\pi_{2n-4}$ to obtain $(n-2,2n-5)(3n-7,3n-5)$. Then we conjugate by $\pi_{(3n-7)/2}$ to obtain $(n-2,2n-3)(3n-7,3n-5)$. Then we conjugate by $\pi_{n-1}$ again to obtain $(n-2,n)(3n-7,3n-5)$. Then we conjugate by $\pi_{2n-3}$ to obtain $(n-2,n)(2n,2n-2)$. Finally, we conjugate by $\pi_{n+1}$ to obtain $(n-2,n)(n+1,n+3)$.

In both cases, we have the permutation $(n-2,n)(n+1,n+3)$. Now we can conjugate by $\pi_1$, and then by $\pi_4$ to obtain $(1,3)(4,6)$. Consider the $4n-8$ strings starting with the 4th string. Then, as in the proof of Lemma~\ref{lem:sameParity4k+1}, we can move the entries of $(4,6)$ to any even numbers between 4 and $4n-5$. Thus we can obtain the permutation $(1,3)(l_b,l_b+2)$. Now we want to move $(1,3)$ to $(l_a, l_a+2)$.

First consider the case when $l_b \le \frac{3n-1}{2}$.
We conjugate $(1,3)(l_b,l_b+2)$ by $\pi_{l_b}$ to obtain $(1,3)(l_b+n-1,l_b+n-3)$. We then conjugate by $\pi_{l_b+n-3}$ to obtain $(1,3)(l_b+2n-6, l_b+2n-4)$. Now by Lemma~\ref{lem:triple8k+5} we can conjugate by $n$-crossing permutations over $S_{(3n-1)/2}$ to move $(1,3)$ to $(l_a,l_a+2)$. Since $l_b+2n-6 > \frac{3n-1}{2}$ the other permutation is not affected, and we are left with $(l_a,l_a+2)(l_b+2n-6, l_b+2n-4)$ We can then conjugate back by $\pi_{l_b+n-3}$ and then by $\pi_{l_b}$ to obtain $(l_a, l_a+2)(l_b, l_b+2)$.

Next, consider the case when $l_a + 2 \le \frac{3n-1}{2} < l_b$. Then by Lemma~\ref{lem:triple8k+5} we can move $(1,3)$ to $(l_a,l_a+2)$ without affecting any other elements.

Finally consider the case when $l_a + 2 > \frac{3n-1}{2}$. Then by Lemma~\ref{lem:triple8k+5} we can move $(1,3)$ to some $(j,j+2)$, where $j < \frac{3n-1}{2}$ and $j$ is a multiple of $n-1$ away from $l_a$. Then we can conjugate by $\pi_{j+2}$ and then by $\pi_j$ to obtain $(j+n-1,j+n+1)$. We can continue this until we reach $(l_a,l_a+2)$.

If $l_a$ is even, we start with $\pi_{(n+3)/2}$, and similarly take conjugations to obtain the permutation $(2,4)(5,7)$. We can obtain $(2,4)(l_b,l_b+2)$ using the $4n-8$ strings starting with the 5th string. Moving $(2,4)$ to $(l_a,l_a+2)$ can be done in the same way as above.
\end{proof}

\subsection{Proof of Theorem~\ref{thm:oddAlexander}}

Now we use these pairs of transpositions to produce the desired braid.

\begin{proof}[Proof of Theorem~\ref{thm:oddAlexander}]
If $n$ is even, the result follows directly from Theorem~\ref{thm:evenAlexander}.

If $n$ is odd, by Theorem~\ref{thm:oddAlexander} we can put the link $L$ in a triple crossing braid form. We consider three cases for when $n$ is odd: $n=4k+3$, $n=8k+5$, and $n=8k+1$.

\noindent
\textbf{Case 1: $n = 4k+3$.}
Take 3-stabilizations until we have at least $3n+5$ strings. It suffices to show that we can obtain any triple-crossing permutation, for then by Lemma~\ref{lem:levelPosition} we can produce any triple-crossing.

Observe that when $n=4k+3$, an $n$-crossing permutation consists of $k$ even-string transpositions and $k+1$ odd-string transpositions, or $k$ odd-string transpositions and $k+1$ even-string transpositions.

Recall that by Lemma~\ref{lem:sameParity4k+3}, we can obtain pairs of odd-string transpositions, and pairs of even-string transpositions, as products of the $n$-crossing permutations. We can obtain one odd-string transposition as follows: start with $\pi_1 = (1,n)(2,n-1)\cdots(\frac{n-1}{2},\frac{n+3}{2})$, which has an odd number of odd-string transpositions and an even number of even-string transpositions. Then we can cancel the even number of even-string transpositions, and all but one of the odd-string transpositions, by multiplying it with pairs of odd-string transpositions and pairs of even-string transpositions. We are left with a single odd-string transposition.

As before, we can focus on the odd strings and see that we can obtain any permutation of the same cycle type by Lemma~\ref{lem:sameCycleType}, since the $(4k+3)$-crossing permutations starting on odd indices function as $(2k+2)$-crossing permutations on the odd strings. Therefore we can obtain any odd-string transposition. We can similarly obtain any even-string transposition. This shows that we can produce any triple-crossing as a sequence of $n$-crossings, so $L$ can be put into an $n$-crossing braid.

\noindent
\textbf{Case 2: $n=8k+5$.} First observe that an $(8k+5)$-crossing permutation is an even permutation (here we mean even in the traditional sense of the word in symmetric groups, namely that it can be written as a product of an even number of transpositions.) Therefore we want to start with a triple-crossing braid with an even number of crossings, and then turn pairs of triple-crossings into sets of $n$-crossings.

Take 3-stabilizations of $L$ until we have at least $3n-2$ strings. If this braid has an odd number of triple crossings at this stage, take one more 3-stabilization so that we have an even number of triple crossings.

We claim we can put these triple-crossings in pairs such that each pair is disjoint, in the sense that they involve 6 distinct strings in total. This is because if two consecutive crossings involve a common strand, then we can put one of them in a pair with a dummy crossing far away, and put the other in another pair which undoes the dummy crossing (Fig.~\ref{fig:disjointPair}).

\begin{figure}[ht]
	\centering
	\includegraphics[scale=0.2]{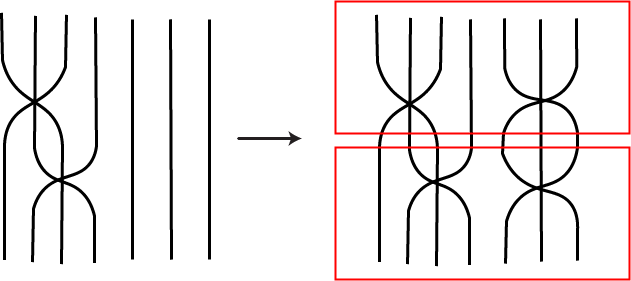}
	\caption{If two consecutive crossings share a common strand, we can put them in two disjoint pairs along with two dummy crossings.}
	\label{fig:disjointPair}
\end{figure}

Now it suffices to show that we can obtain any permutation corresponding to a disjoint pair of triple-crossings. By Lemma~\ref{lem:sameParity4k+1}, we can obtain pairs of transpositions of the same parity. By Lemma~\ref{lem:pair8k+5}, we can obtain any pair of an even-string transposition and an odd-string transposition, corresponding to a disjoint pair of triple crossings. Thus by Lemma~\ref{lem:levelPosition} we can produce any pair of disjoint triple-crossings using $n$-crossings. Hence $L$ can be put into a $n$-crossing braid.

\noindent
\textbf{Case 3: $n=8k+1$.} First observe that a $(8k+1)$-crossing permutation consists of $2k$ even-string transpositions and $2k$ odd-string transpositions. Therefore we want to start with a triple-crossing braid with an even number of even-string crossings, and an even number of odd-string crossings.

Take 3-stabilizations of $L$ until we have at least $3n-2$ strings. If the number of even-string crossings and odd-string crossings are both even at this stage, then it is in the desired form. If both numbers are odd, then we can take two 3-stabilizations, which will increase both numbers by one.

The remaining case is when the number of one set of crossings is odd, and the number of the other set of crossings is even. In this case, we consider two possibilities. First suppose the number of strings is even. Then note that a 3-stabilization on the second string will increase the number of even-string crossings, and a 3-stabilization on the second last string will increase the number of odd-string crossings (Fig.~\ref{fig:3-stabilization2}). Hence if the parity is off by one, we can always find the appropriate 3-stabilization.

\begin{figure}[ht]
	\centering
	\includegraphics[scale=0.3]{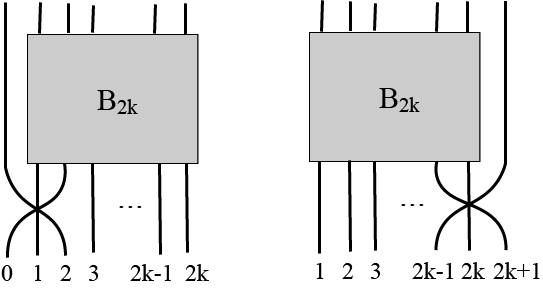}
	\caption{A 3-stabilization on the second string will increase the number of even-string crossings, and a 3-stabilization on the second last string will increase the number of odd-string crossings.}
	\label{fig:3-stabilization2}
\end{figure}

Now suppose the number of strings is odd. Then both 3-stabilizations increase the number of even-string crossings. But once we take this stabilization, the number of even-string crossings and the number of odd-string crossings will have the same parity. We can therefore put it in the desired form with at most two more 3-stabilizations.

Hence we can put $L$ into a triple-crossing braid with an even number of even-string crossings, and an even number of odd-string crossings. As before we can put these crossings into pairs that are disjoint. Note that here we must also put them in pairs such that they have the same parity. We can then produce each pair of crossings since we can obtain pairs of transpositions with the same parity by Lemma~\ref{lem:sameParity4k+1}. Therefore we can produce each pair of triple-crossings with the same parity using the $n$-crossings, so $L$ can be put into a $n$-crossing braid form.
\end{proof}

\section{Bounds on Braid Indices}
\label{sec:braidIndices}

In this section we find relationships between the multi-crossing braid indices. Many arguments in this section have been simplified; details can be found in the original thesis.

Let $\beta_n(L)$ be the minimum number of strings necessary to represent the link $L$ as an $n$-crossing braid. We define $\beta_n(L) = \infty$ if $L$ cannot be represented as an $n$-crossing braid. Note this happens if and only if $L$ is a knot and $n$ is odd. A simple observation tells us the following.

\begin{theorem}
Let $L$ be a link. For any $n \ge 2$, we have
\[
\beta_2(L) \le \beta_n(L).
\]
\end{theorem}

\begin{proof}
Observe that any multi-crossing braid can be turned into a double crossing braid with the same number of strings, by breaking up each multi-crossing into double crossings.
\end{proof}

\subsection{Even multi-crossing braids}

For even $n$ we have the following results.

\begin{theorem}
\label{thm:braidEven}
Let $L$ be a link that is not an unlink. Let $n \le 202$, and $m \ge n+2$.
\begin{itemize}
\setlength\itemsep{0.3em}

\item[(i)]
If $n=4k+2$, then we have $\beta_n(L) \le \beta_m(L)$.

\item[(ii)]
If $n=4k$, then we have $\beta_n(L) \le \beta_m(L) + 1$.

Moreover, if $m=4k$ or $m=4k+1$, then we have $\beta_n(L) \le \beta_m(L)$.
\end{itemize}
\end{theorem}

\begin{proof}
\text{}
\begin{itemize}
\setlength\itemsep{0.3em}

\item[(i)]
We can show through computations in Mathematica that for all $n\le 202$ with $n=4k+2$, the $n$-crossing permutations over $S_{n+2}$ generate $S_{n+2}$. (See the original thesis for the code.) By Lemma~\ref{lem:levelPosition}, this means that we can obtain every double crossing as a product of $n$-crossings if we have $n+2$ strings.

Let $m \ge n+2$. Since an $m$-crossing requires $m$ strings, and $L$ is not an unlink, we have $\beta_m(L) \ge n+2$. We can decompose an $m$-crossing braid into a double crossing braid with the same number of strings. Then each double crossing can be turned into a product of $n$-crossings, so we have an $n$-crossing braid with $\beta_m(L)$ strings. This means that there is no need for the extra stabilization, so we have $\beta_n(L) \le \beta_m(L)$ for all such $m \ge n$.

\item[(ii)]
We can again show with Mathematica that for all such $n \le 200$, the $n$-crossing permutations over $S_{n+2}$ generate $A_{n+2}$, the alternating group on $n$ elements. The first inequality follows as in (i), except in this case we may have to take a stabilization to ensure that the number of double crossings is even.

Observe that if $m = 4k$ or $m=4k+1$, then an $m$-crossing breaks down into an even number of double crossings. Thus the double crossing braid must have an even number of crossings. Therefore $\beta_n(L) \le \beta_m(L)$ for all such $m \ge n$.
\end{itemize}
\end{proof}

Hence we have the following relationships:
\begin{align*}
\beta_2&(L) \le \beta_6(L) \le \beta_{10}(L) \le \cdots \le \beta_{202}(L) \le \beta_{206}(L) \\
&\updownle \hspace{25pt} \updownle \hspace{32pt} \updownle \hspace{55pt} \updownle \\
\beta_4&(L) \le \beta_8(L) \le \beta_{12}(L) \le \cdots \le \beta_{204}(L).
\end{align*}
And the following:
\[
\beta_4(L) \le \beta_6(L) + 1, \hspace{5pt} \beta_8(L) \le \beta_{10}(L) + 1, \hspace{5pt} \dots, \hspace{5pt} \beta_{200}(L) \le \beta_{202}(L) + 1.
\]

Note that the inequalities need not end at 202; this is an arbitrary choice on how far to go with the Mathematica computation. It would however be interesting to ask whether there is a general argument that could extend the inequalities to all even $n$.

\begin{example}
The inequality $\beta_{4k} \le \beta_{4k+2} + 1$ is strict for any link $L$ with the property that $\beta_2(L) \ge 4k+4$, and the number of crossings when it realizes the double crossing braid index is odd. We know this because Markov moves (\cite{markov}) preserve the parity of $b+c$, where $b$ is the number of strings and $c$ is the number of crossings. A conjugation does not alter either $b$ or $c$, and a stabilization increases both $b$ and $c$ by 1. This means that given such a link, it cannot be represented as any double crossing braid with $\beta_2(L)$ strings and an even number of double crossings. In order to turn it into a $4k$-crossing braid we must therefore take a stabilization to make the number of crossings even, so $\beta_{4k}(L) = \beta_2(L) + 1$. We can of course realize $L$ as a $(4k+2)$-crossing braid with $\beta_2(L)$ strings, so $\beta_{4k+2}(L) = \beta_2(L) = \beta_{4k}(L) - 1$. An artificial example of such a link is a split link consisting of the trefoil knot and an unlink with $4k+2$ components.
\end{example}

We also have inequalities that hold for infinitely many even $n$.

\begin{theorem}
\label{thm:braidEvenAll}
Let $L$ be a link that is not an unlink.
\begin{itemize}
\setlength\itemsep{0.3em}

\item[(i)]
Let $n=4k+2$. Then for any $m \ge 3n+1$, we have $\beta_n(L) \le \beta_m(L)$.

\item[(ii)]
Let $n=4k$. Then for any $m \ge 3n+1$, we have $\beta_n(L) \le \beta_m(L) + 1$.

Moreover, if $m=4k$ or $m=4k+1$, then we have $\beta_n(L) \le \beta_m(L)$.
\end{itemize}
\end{theorem}

\begin{proof}
(i) The proof of Theorem~\ref{thm:evenAlexander} required at least $3n+1$ strings. This means that for $m \ge 3n+1$, we can turn an $m$-crossing braid into an $n$-crossing braid with the same number of strings, so $\beta_n(L) \le \beta_m(L)$. (ii) can be proved similarly.
\end{proof}

\begin{corollary}
Let $L$ be a link that is not an unlink. Then for all even $n$,
\[
\beta_n(L) \le \beta_{3n+1}(L).
\]
\end{corollary}

\begin{proof}
After noting that when $n=4k$, $3n+1$ can be written in as $8k'+1$ or $8k'+5$ for some $k'$, the result follows directly from Theorem~\ref{thm:braidEvenAll}.
\end{proof}

\subsection{Odd multi-crossing braids}

For links with two or more components, we can consider braid indices for odd $n$. First consider the case when $n=3$.

\begin{theorem}
\label{thm:braidn=3}
Let $L$ be a link with two or more components that is not an unlink. Then
\[
\beta_3(L) =
\begin{cases}
3 &\mbox{ if } \beta_2(L) = 2; \\
\beta_2(L) &\mbox{ otherwise.}
\end{cases}
\]
\end{theorem}

\begin{proof}
The proof of Theorem~\ref{thm:tripleAlexander} required that the double crossing braid have at least 3 strings, and gave us a triple crossing braid with the same number of strings.
\end{proof}

By this we can see that $\beta_3(L) \le \beta_n(L)$ for any odd $n$. However, this can also be seen by noting that any $n$-crossing can be decomposed into triple crossings; any $n$-crossing is a permutation of even strings and a permutation of odd strings, each of which can be realized as a product of triple crossings.

\begin{theorem}
\label{thm:braidOdd}
Let $L$ be a link with two or more components that is not an unlink. Let $n \le 205$, and $m \ge n+3$.
\begin{itemize}
\setlength\itemsep{0.3em}

\item[(i)]
If $n=4k+3$, then we have $\beta_n(L) \le \beta_m(L)$.

\item[(ii)]
If $n=8k+5$, then we have $\beta_n(L) \le \beta_m(L) + 1$.

Moreover, if $m=4k+1$, then we have $\beta_n(L) \le \beta_m(L)$.

\item[(iii)]
If $n=8k+1$, then we have $\beta_n(L) \le \beta_m(L) + 3$.

Moreover, if $m=8k+1$, then we have $\beta_n(L) \le \beta_m(L)$.
\end{itemize}
\end{theorem}

\clearpage

\begin{proof}
\text{}
\begin{itemize}
\setlength\itemsep{0.3em}

\item[(i)]
We can show through computations in Mathematica that for all $n\le 203$ with $n=4k+3$, the $n$-crossing permutations over $S_{n+3} = S_{4k+6}$ generate $S_{2k+3} \times S_{2k+3}$. Each $S_{2k+3}$ corresponds to permutations of odd strings and permutations of even strings. By Lemma~\ref{lem:levelPosition}, this means that we can obtain every triple crossing as a product of $n$-crossings if we have $n+2$ or more strings. The rest of the proof follows as in the proof of Theorem~\ref{thm:braidEven} (i), but by decomposing the $m$-crossing braid into triple crossings.

\item[(ii)]
We can again show with Mathematica that for all such $n \le 205$, the $n$-crossing permutations $S_{n+3} = S_{8k+8}$ generate half of $S_{4k+4} \times S_{4k+4}$. Since we know that the $n$-crossing permutations must be even, this shows that the $n$-crossing permutations generate all pairs of triple crossings permutations. Thus, if we have $n+3$ or more strings, we can obtain every pair of triple crossings as a product of $n$-crossings. The rest of the proof follows as in the proof of Theorem~\ref{thm:braidEven} (ii), but by decomposing the $m$-crossing braid into triple crossings.

\item[(iii)]
We can again show with Mathematica that for all such $n \le 201$, the $n$-crossing permutations $S_{n+3} = S_{8k+4}$ generate $A_{4k+2} \times A_{4k+2}$. This shows that the $n$-crossing permutations generate all pairs of even-string transpositions, and all pairs of odd-string transpositions. Thus, if we have $n+3$ or more strings, we can obtain every pair of odd triple crossings and every pair of even triple crossings as a product of $n$-crossings. The rest of the proof follows as in (ii).
\end{itemize}
\end{proof}

As for the case with $n$ even, we suspect that the inequalities can be extended for all $n$.

The following inequalities hold for infinitely many $n$. This is done as in the proof of Theorem~\ref{thm:braidEvenAll}, by checking the number of strings that were necessary for the proof of Theorem~\ref{thm:oddAlexander}.

\begin{theorem}
\label{thm:braidOddAll}
Let $L$ be a link with two or more components that is not an unlink.
\begin{itemize}
\setlength\itemsep{0.3em}

\item[(i)]
Let $n=4k+3$. Then for any $m \ge 3n+5$, we have $\beta_n(L) \le \beta_m(L)$.

\item[(ii)]
Let $n=8k+5$. Then for any $m \ge 3n-2$, we have $\beta_n(L) \le \beta_m(L) + 1$.

Moreover, if $m=4k+1$, then we have $\beta_n(L) \le \beta_m(L)$.

\item[(iii)]
Let $n=8k+1$. Then for any $m \ge 3n-2$, we have $\beta_n(L) \le \beta_m(L) + 3$.

Moreover, if $m=8k+1$, then we have $\beta_n(L) \le \beta_m(L)$.
\end{itemize}
\end{theorem}

Finally we present an inequality that holds for all $n \ge 8$, even or odd.

\begin{corollary}
Let $L$ be a link that is not an unlink. Then for all $n \ge 8$,
\[
\beta_n(L) \le \beta_{4n-3}(L).
\]
\end{corollary}

\begin{proof}
Observe that when $n \ge 8$, we have $4n-3 \ge 3n+5$. Observe the following about $4n-3$. When $n=4k$, we have $4n-3 = 16k-3 = 8(2k-1) + 5$. When $n=8k+5$, we have $4n-3 = 32k+20-3 = 8(4k+2) + 1$. Finally, when $n=8k+1$, we have $4n-3 = 32k+4-3 = 8(4k) + 1$. Then the result follows directly from Theorem~\ref{thm:braidEvenAll} and Theorem~\ref{thm:braidOddAll}.
\end{proof}

\FloatBarrier 

\end{document}